\date{22 July 2020}

\documentclass[10pt,reqno]{amsart}
\usepackage{latexsym,amsmath,amsfonts,amscd,amssymb,url}
\usepackage{graphics}
\usepackage{xcolor}
\newtheorem{theorem}{Theorem}[section]
\newtheorem{lemma}[theorem]{Lemma}
\newtheorem{proposition}[theorem]{Proposition}
\newtheorem{corollary}[theorem]{Corollary}

\theoremstyle{remark}
\newtheorem{remark}[theorem]{Remark}

\newcommand{\dd}{\mathrm{d}}

\newcommand{\CC}{{\mathbb C}}

\newcommand{\QQ}{{\mathbb Q}}

\newcommand{\ZZ}{{\mathbb Z}}

\title{On the genesis of BBP formulas}


\author[D. Barsky]{Daniel Barsky}
 \address{7 rue La Condamine, 75017 Paris, France}
 \email{barsky.daniel@orange.fr}

 \author[V. Mu\~{n}oz]{Vicente Mu\~{n}oz}
 \address{Departamento de Algebra, Geometr\'{\i}a y Topolog\'{\i}a, Universidad de M\'alaga,
 Campus de Teatinos, s/n, 29071 Malaga, Spain}
 \email{vicente.munoz@uma.es}
 
 \author[R. P\'{e}rez-Marco]{Ricardo P\'{e}rez-Marco}
 \address{CNRS, IMJ-PRG, Univ. de Paris, B\^at. Sophie Germain, Case 7012, 75205-Paris Cedex 13, France}
 \email{ricardo.perez-marco@math.imj-prg.fr}

\begin{document}

\begin{abstract}
We present a general procedure to generate infinitely many  BBP and BBP-like formulas for the simplest transcendental numbers. This provides some insight and a better understanding into their nature. In particular, we can derive the main known BBP formulas for $\pi$. We can understand why many of these formulas are rearrangements of each other. 
We also understand better where some null BBP formulas representing $0$ come from. We also explain what is the 
observed relation between some BBP formulas for $\log 2$ and $\pi$, that are obtained by taking real and 
imaginary parts of a general complex BBP formula. Our methods are elementary, but 
motivated by transalgebraic considerations, and offer a new way to obtain and to search many 
new BBP formulas and, conjecturally, to better understand transalgebraic 
relations between transcendental constants.
\end{abstract}

\maketitle

\section{Introduction}

More than 20 years ago, D.H. Bailey, P. Bowein and S. Plouffe (\cite{BBP}) presented an efficient 
algorithm to compute deep binary or hexadecimal  digits of $\pi$ 
without the need to compute the previous ones. Their algorithm is based on a 
series representation for $\pi$ given by a formula discovered by S. Plouffe, 
 \begin{equation}\label{eqn:a}
\pi = \sum_{k=0}^{+\infty} \frac{1}{16^k} \left (\frac{4}{8k+1}-\frac{2}{8k+4}-\frac{1}{8k+5}-\frac{1}{8k+6}\right ).
 \end{equation}
Formulas of similar form for other transcendental constants were known from long time ago, like the classical formula for
$\log 2$, that was known to J. Bernoulli,
 \begin{equation*}
\log 2 =\sum_{k=1}^{+\infty} \frac{1}{2^k}\,\frac{1}{k} \, .
 \end{equation*}
The reader can find in \cite{BBMW} an illustration of the way to extract binary digits from this type of formulae.

Many new formulas of this type, named BBP formulas, have been found for $\pi$ and other higher transcendental 
constants in the last decades (see \cite{B}, \cite{T}).
Plouffe's formula, and others for $\pi$, can be derived  using integral periods (as in \cite{BBP}), 
or more directly using polylogarithm ladder relations at precise algebraic values (as in \cite{Br}), which 
can be viewed as generalizations of Machin-St\"ormer relations (see \cite{Sto1} and \cite{Sto2}) 
for rational values of the 
arctangent function, and taking its Taylor series expansions. In particular, we can recover 
in that way Bellard's formula  (see Bellard's webpage \cite{Be}), that seems to be the most efficient one for the purpose of computation 
of deep binary digits of $\pi$ (see Remark \ref{rem:efficient}),
 \begin{equation}\label{eqn:c}
\pi = \frac{1}{2^6}\sum_{k=0}^{+\infty} \frac{(-1)^k}{2^{10k}} \left (-\frac{2^5}{4k+1}
-\frac{1}{4k+3}+\frac{2^8}{10k+1}-\frac{2^6}{10k+3}-\frac{2^2}{10k+5}-\frac{2^2}{10k+7}+\frac{1}{10k+9}\right ).
 \end{equation}

Many of these formulas are rearrangements of each other, or related by null BBP formulas that represent $0$. The 
origin of null BBP formulas is somewhat mysterious. 
Most of the formulas of this sort have been found by extensive computer search over parameter space 
using the PSLQ algorithm to detect integer relations. 
So their true origin and nature remainded somewhat mysterious. As the authors of \cite{BBP} explain:


\textit{We found the identity by a combination of inspired guessing and extensive searching using the PSLQ integer relation algorithm.}

and in \cite{BB2}


\textit{This formula (\ref{eqn:a}) was found using months of PSLQ computations after corresponding but simpler n-th digit 
formulas were identified for several other  constants, including $\log (2)$. This is likely the first instance in history 
that a significant new formula for $\pi$ was discovered by a computer.}


We note also the  observed mysterious numerical relation of BBP formulas for $\pi$ and $\log (2)$.


For the purpose of computation of \textit{all} digits of $\pi$ up to a certain order, there are more efficient 
formulas given by rapidly convergent series of a modular nature, initiated by S. Ramanujan (\cite{R}), that are at the origin
of Chudnoskys' algorithm based on Chudnovskys' formula (see \cite{CC})
 \begin{equation*} 
\frac{1}{\pi}= 12 \sum_{k=0}^{+\infty} \frac{(-1)^k (6k)! (545140134k+13591409)}{(3k)!(k!)^3(640320)^{3k+3/2}}\, .
 \end{equation*}
Other methods of algorithmic nature include the Borwein quartic algorithm for $\pi$ (see \cite{BBo})
that approximately quadruples the number of correct digits with each iteration, and
the Borwein nonic algorithm for $\pi$ that approximately yields nine-times the number of correct digits.

A general BBP formula as defined in \cite{BC} for the constant $\alpha$ is a series of the form
$$
\alpha = Q(d,b,m,\mathbf{A})=\sum_{k=0}^{+\infty} \frac{1}{b^k}\sum_{l=1}^m \frac{a_l}{(km+l)^d}\, ,
$$
where $b, d,m$, are integers, $b\geq 2$, and $\mathbf{A}=(a_1, a_2,\ldots , a_m)$ is an integer vector. 
The integer $d\geq 1$ is the degree of the formula. The classical BBP formula (\ref{eqn:a}) 
and Bellard formula (\ref{eqn:c}) are of degree $1$. We study in this article formulas of degree $1$. 
It would be interesting to extend the present results to get higher degree formulas.
The integer $b$ is called the base of the BBP formula, and 
digits in base $b$ can be computed efficiently. Particular attention has been given to base $b=2^n$ formulas, as they 
are useful in computing binary digits. 
They are called \textit{binary} BBP formulas. While there are both base $2$ and base $3$ BBP formulas for
some constants like $\pi^2$ (see \cite{BBG}), no base $3$ formula for $\pi$ is known.

More generally, we can define BBP-like formulas to be of the general form
 \begin{equation}\label{eqn:d}
\alpha = Q(r_0, r_1,d,b,m,\mathbf{A})= 
r_0+ r_1 \sum_{k=0}^{+\infty} \frac{1}{b^k}\sum_{l=1}^m \frac{a_l}{(km+l)^d}\, ,
 \end{equation}
where $r_0$ and $r_1$ are rational numbers. These more general BBP-like formulas have potentially 
similar computational applications.


But the interest of these formulas is also theoretical. A \textit{normal number} in base $b\geq 2$ is an irrational 
number $\alpha$ such that its expansion in base $b$ contains any string of $n$ consecutive digits with frequency $b^{-n}$. 
These numbers were introduced in 1909 by \'E. Borel in an article where he proved that Lebesgue almost every number 
is normal in any base $b\geq 2$ (\cite{Bo}, and the survey \cite{Q}).
This result is a direct application of Birkhoff's Ergodic Theorem to the dynamical system given by the
transformation $T: \mathbb{T}\to \mathbb{T}$, multiplication by the base $T(x)=bx$ modulo $1$, where $\mathbb{T}=\mathbb{R}/\mathbb{Z}$.
The transformation $T$ preserves the Lebesgue measure which is an ergodic invariant measure. It is not difficult to construct explicit 
normal numbers, and numbers that are not normal, but there is no known example of ``natural'' transcendental constant that is normal in 
every base. It is conjectured that this holds for $\pi$ and other natural transcendental constants, but this remains an open question. It is 
not even known if a given digit appears infinitely often in the base $10$ expansion for $\pi$. We note recent
results \cite{BC2,BM} where the normality of certain class of constants has been proved, yet not including $\pi$.


An approach to prove normality in base $b$ for any transcendental constant which admits 
a BBP formula in base $b$ is proposed in \cite{BC}. The criterion, named ``Hypothesis A'', seems related to 
Furstenberg's ``multiplication by $2$ and $3$'' conjecture (see \cite{L}). Only a very particular class of period-like numbers have 
BBP formulas (for instante, as mentioned before, $\pi^2$ does). 
It is also natural to investigate the class of numbers with a BBP or BBP-like representation.


The main goal of this article is to present a general procedure to generate the most basic BBP and BBP-like formulas of degree $1$ 
that correspond to the simplest transcendental numbers $ \log p$ and $\pi$. With this new procedure we derive the classical formulas, like 
Bailey-Borwein-Plouffe or Bellard formulas, and understand better their origin, in particular the origin of null formulas, and the relation 
of BBP formulas for $\log 2$ and $\pi$ that correspond to take the real or imaginary parts of the same complex formula. We also understand 
better the redundancy of rearrangements in these formulas, and the method provides a tool to search for more formulas with a more conceptual 
approach. Although we do not find new BBP formulas, we recover the most important ones and we believe that the method presented can be further developed to discover new ones. We plan to carry this out in the future.


The procedure to generate BBP formulas is elementary and is motivated by considering 
the bases for first order asymptotics at infinite of Euler Gamma function and higher 
Barnes Gamma functions and the transalgebraic considerations that play an important role in  \cite{MPM2} (see also \cite{MPM1}). 
To construct these asymptotic bases, we consider the family iterated integrals 
of $\frac{1}{s}$ defined by $I_0(s)=\frac{1}{s}$, and for $n\geq 0$,
$$
I_{n+1}(s) =\int_1^s I_n(u) \, \dd u =\ldots =
\int_1^s \int_1^{u_n}\ldots \int_1^{u_{0}} \frac{1}{u_0} \, \dd u_0\ldots \dd u_{n-1}\, \dd u_n \, .
$$
It is elementary to check by induction that 
$$
I_n(s) =A_n(s) \log s +B_n(s),
$$ 
where $A_n, B_n \in \QQ[s]$ are polynomials with rational coefficients, with $\deg A_n =\deg B_n =n-1$, and
$$
A_n(s)=\frac{s^{n-1}}{(n-1)!}\, ,
$$ 
we have
\begin{theorem} \label{thm_main}
Let $s\in \CC$, 
$|s-1|< 1$,  or $|s-1|=1$ and $n\geq 2$. We have
$$
I_n(s)= \sum_{j=0}^{+\infty} \frac{(1-s)^{j+n}}{n! \binom{j+n}{n}} =  \frac{s^{n-1}}{(n-1)!} \log s +B_n(s),
$$
or
$$
\log s =\frac{(1-s)^n}{ns^{n-1}}\sum_{j=0}^{+\infty} 
\frac{(1-s)^{j}}{\binom{j+n}{j}} - \frac{(n-1)!}{s^{n-1}} B_n (s) .
$$
\end{theorem}

Since $B_n(s)$ has rational coefficients, we can take $s=1-\frac1b$ and we get a BBP-like formula 
for $\log s$. Taking suitable complex values for $s$, and separating real and imaginary parts, we also
obtain BBP and BBP-like formulas for $\pi$. We prove that formulas for different values of $n$ provide non-obvious 
rearrangements of the summations, which in part explains the rich ``rearrangement algebra'' of BBP formulas.


We recover many formulas with this procedure. For instance, all the formulas of $\log 2$ appearing in Wikipedia \cite{Wikipedia}
are given in (\ref{eqn:log2-1})--(\ref{eqn:log2-11}). We also get the following classical formulas:
\begin{align*} 
 \log 2 &=\frac{5}{6}-\sum_{k=1}^{+\infty}\frac{1}{2^k} \left (\frac{1}{k}-\frac{3}{k+1}+\frac{3}{k+2}-\frac{1}{k+3} \right ), \\
 \log 2 =& \frac23 + \sum_{k=1}^{+\infty} \frac{1}{16^k} \left (\frac{2}{8k}+\frac{1}{8k+2}+\frac{1/2}{8k+4}+\frac{1/4}{8k+6} \right ) , \\
 \pi =& \frac{8}{3} +4\sum_{k=1}^{+\infty} \left (\frac{1}{4 k+1} -\frac{1}{4 k+3} \right ),\\
 \pi=&   \sum_{k=0}^\infty \frac{1}{16^{k}} \left(
\frac{2}{8k+1} +\frac{2}{8k+2}+\frac{1}{8k+3} -\frac{1/2}{8k+5}-\frac{1/2}{8k+6} -\frac{1/4}{8k+7}
\right). 
\end{align*}

Also combining our formulas we can get some null formulas representing $0$, as for example the following one appearing in \cite{BBP}

\begin{equation}
0= \sum_{k=0}^{+\infty}  \frac{1}{16^{k}} \left( 
\frac{-8}{8k+1}+\frac{8}{8k+2} +\frac{4}{8k+3}+\frac{8}{8k+4} +\frac{2}{8k+5} +\frac{2}{8k+6}- \frac{1}{8k+7}  \right ).
\end{equation}

This gives some explanations of the mysteries mentioned before.
For example, formulas for $\log 2$ and $\pi$ are related by taking real 
and imaginary parts of formulas for complex values for $s$, 
for example for $s=\frac{1+i}{2}$. Null formulas can appear when comparing our formulas for different complex values of $s$ taking real or imaginary parts. For example
for $s=1/2$ and $s=\frac{1+i}{2}$ we do get the previous null formula. It is natural to ask if all null BBP formulas of degree $1$ can be obtaining combining formulas from
Theorem \ref{thm_main} for different values of $s$.

Certainly, we also recover the classical BBP formula (\ref{eqn:a}) and Bellard formula (\ref{eqn:c}). 

\begin{remark}\label{rem:efficient}
We can measure the \emph{efficiency} of a BBP-like formula (\ref{eqn:d}) for computing a number $\alpha$
as ${\bar m}/\log b$, where $\bar m$ is the number of non-zero coefficients in $\mathbf{A}=(a_1,a_2,\ldots, a_m)$, 
as this measures the number of non-zero summands for going to the next step in the digit computation.
Binary BBP formulas, that is when $b=2$, are of special relevance, since they allow to compute $\alpha$ in
binary form. In that case, we can take the logarithm in base $2$. The efficiency of (\ref{eqn:a}) is $1$, whereas
the efficiency of (\ref{eqn:c}) is $7/10$, a $43\%$ faster.
\end{remark} 

The techniques of this article extend to other bases of iterated functions that we will discuss in future articles. 
We hope that our approach can be useful in finding more efficient BBP-formulas for $\pi$ by more powerful 
algebraic computer search algorithms.

\subsection*{Acknowledgements} We are very grateful to the anonymous referee that has made a large number of interesting
suggestions to improve the exposition. We thank also Tomohiro Yamada for pointing out several corrections. 
The second author was partially supported by Project MINECO (Spain) PGC2018-095448-B-I00.

\section{Laplace-Hadamard regularization of polar parts}\label{sec:Laplace-Hadamard}

The Laplace-Hadamard regularization is related to work in \cite{MPM1} and \cite{MPM2}.


For each $n\geq 0$ we define the polynomials $P_0=0$, and for $n\geq 1$,
$$
P_n(s,t) =\sum_{k=0}^{n-1} \frac{(1-s)^k}{k!} t^k \, . 
$$
We also define the iterated primitives of $1/s$ defined by $I_0(s)=\frac{1}{s}$, and for $n\geq 0$,
$$
I_{n+1}(s) =\int_1^s I_n(u) \, \dd u =\ldots =
\int_1^s \int_1^{u_n}\ldots \int_1^{u_{0}} \frac{1}{u_0} \, \dd u_0\ldots \dd u_{n-1}\, \dd u_n \, .
$$
We call the integrals $I_n(s)$ the Laplace-Hadamard regularization or the Laplace-Hadamard transform of $1/t^n$. 
The functions $I_n(s)$ are holomorphic functions in $\CC-]\!-\infty, 0]$ and have an isolated singularity at $0$
with non-trivial monodromy when $n\geq 1$. We have a single integral expression for $I_n(s)$ as 
a Laplace-Hadamard regularization:

\begin{proposition} \label{prop:polar_regularization}
For $n\geq 0$ and $\Re s >0$, or $\Re s =0$ and $n\geq 2$, we have
$$  
 I_n(s)= (-1)^n \int_0^{+\infty} \frac{1}{t^n} \left ( e^{-st}- P_n(s,t) e^{-t}\right ) \, \dd t \, .
$$
\end{proposition}

\begin{proof}
For $n=0$ we have
$$
\int_0^{+\infty} e^{-st} \, \dd t =\frac{1}{s}\, ,
$$
and by induction we get the result integrating on the variable $s$ between $1$ and $s$,
$$ 
  I_{n+1} (s)=\int_1^s \int_0^{+\infty} \frac{1}{t^n} \left ( e^{-ut}- P_n(u,t) e^{-t}\right ) \, \dd t \, \dd u \, ,
$$
and using that 
\begin{align*}
 \int_1^s e^{-ut} \dd u &=-\frac1t( e^{-st}-e^{t}), \\
 \int_1^s P_{n}(u,t) \dd u &= -\frac1t(P_{n+1}(s,t)-1).  
\end{align*}
\end{proof}

Note that we have $P_n(s,t)\to e^{(1-s)t}$ when $n\to +\infty$ uniformly on compact sets, and $P_n(s,t)$ is 
the $n$-th order jet of $e^{(1-s)t}$ at $t=0$.  
So for $t\to 0$ we have
$$
e^{-st}- P_n(s,t) e^{-t} =O (t^n)  .
$$
For $n=0$ we get the elementary integral
 \begin{equation*}
 I_0(s)=\int_0^{+\infty} e^{-st} \dd t =\frac{1}{s} \, .
 \end{equation*}
For $n=1$ we get the old Frullani integral (\cite{BM} p.98) 
 \begin{equation*}
 I_1(s)=-\int_0^{+\infty}  \frac{1}{t}(e^{-st}-e^{-t})\, \dd t =\log s \, .
 \end{equation*}

We have
\begin{align*}
I_1 (s) &= \log s,   \\ 
I_2 (s) &= s \log s -(s-1),  \\ 
I_3 (s) &=\frac{s^2}{2}\log s -\frac{1}{4}(s-1)(3s-1),  \\ 
I_4 (s) &=\frac{s^3}{6}\log s -\frac{1}{36}(s-1)(11s^2-7s+2),  \\ 
I_5 (s) &=\frac{s^4}{24}\log s  -\frac{1}{288} (s-1)(25 s^3 - 23 s^2 + 13 s - 3).  
\end{align*}

A simple induction shows
\begin{proposition}\label{prop:end4}
We have
$$
I_n(s) =A_n(s) \log s +B_n(s),
$$ 
where $A_n, B_n \in \QQ[s]$ are polynomials, with $\deg A_n =\deg B_n =n-1$, and
$$
A_n(s)=\frac{s^{n-1}}{(n-1)!}\, .
$$ \hfill $\Box$
\end{proposition}


Regarding the polynomials $B_n$, 
the relation $I'_{n+1}(s)=I_n(s)$ shows that we have
\begin{equation}\label{eq:recurrence_B}
B'_{n+1}(s)=B_n(s)-\frac{s^{n-1}}{n!} \, .
\end{equation}
This equation with the condition $B_{n+1}(1)=0$ determines $B_{n+1}$ uniquely from $B_n$.

We have a formula for $B_n$ (see \cite{MMR}, where $I_n(s)= f_{n-1}(x)$ with $x=s-1$, and \cite{M}):

\begin{proposition} \label{prop:B_polynomials}
We have for $n\geq 0$,
$$
B_{n+1}(s) = -\frac{1}{n!} \sum_{k=1}^{n} \binom{n}{k} (H_{n}-H_{n-k}) (s-1)^{k} \, ,
$$
where $H_0=0$ and $H_n=1+\frac12+\frac13+\ldots +\frac1n$ are the Harmonic numbers.
\end{proposition}

\begin{proof}
 The formula holds for $n=0$ and it satisfies $B_{n+1}(1)=0$ and the recurrence relation:
 \begin{align*}
  B'_{n+1}(s) &= -\frac{1}{n!} \sum_{k=1}^{n} k\binom{n}{k} (H_{n}-H_{n-k}) (s-1)^{k-1}\\
  &=-\frac{1}{(n-1)!} \sum_{k=1}^{n} \binom{n-1}{k-1} (H_{n}-H_{n-k}) (s-1)^{k-1}\\
  &= -\frac{1}{(n-1)!} \sum_{k=1}^{n-1} \binom{n-1}{k-1} (H_{n-1}-H_{n-1-(k-1)}) (s-1)^{k-1}\\
  & \ \ \ -\frac{1}{(n-1)!} \sum_{k=1}^{n}\frac1n (H_{n}-H_{n-1}) (s-1)^{k-1}\\
  &= B_n(s)-\frac{1}{n!} \sum_{k=1}^{n} \binom{n-1}{k-1} (s-1)^{k-1}\\
  &= B_n(s)-\frac{1}{n!} ((s-1)+1)^{n-1}\\
  &= B_n(s)-\frac{s^{n-1}}{n!} \, .
 \end{align*}
\end{proof}

Now we  prove:

\begin{lemma} \label{lemma:key}
For $n\geq 1$,
$$
I_{n+1}(0)=B_{n+1}(0)=\frac{(-1)^{n+1}}{n \cdot n!} \, . 
$$
\end{lemma}

We first establish a useful integral representation for harmonic numbers

\begin{lemma}
$$
H_n =\int_0^{+\infty} \frac{e^{-t}-e^{-(n+1)t}}{1-e^{-t}} \, \dd t\, .
$$
\end{lemma}

\begin{proof}
We have
$$
H_n=\sum_{k=1}^{n} \frac1k = \sum_{k=1}^{n} \int_0^{+\infty} e^{-kt} \, \dd t =\int_0^{+\infty} \frac{e^{-t}-e^{-(n+1)t}}{1-e^{-t}} \, \dd t \, .
$$
\end{proof}

From this it follows

\begin{lemma}
$$
\sum_{k=0}^n \binom{n}{k} (-1)^k H_{n-k} =\frac{(-1)^{n+1}}{n}\, .
$$
\end{lemma}

\begin{proof}
\begin{align*}
\sum_{k=0}^n \binom{n}{k} (-1)^k H_{n-k} &= \int_0^{+\infty} \frac{\left (\sum_{k=0}^n \binom{n}{k} (-1)^k  \right ) e^{-t}-
\left (\sum_{k=0}^n \binom{n}{k} (-1)^k e^{-(n-k)t}\right ) e^{-t} }{1-e^{-t}} \, \dd t \\
&= \int_0^{+\infty} \frac{\left (1-1 \right )^n e^{-t}-
\left (-1+ e^{-t}\right )^n e^{-t} }{1-e^{-t}} \, \dd t \\
&= (-1)^{n+1} \int_0^{+\infty} (1-e^{-t})^{n-1}e^{-t} \, \dd t \\
&= (-1)^{n+1}\int_0^1 x^{n-1} \, \dd x \\
&=\frac{(-1)^{n+1}}{n}\, .
\end{align*}
\end{proof}

Now we can prove Lemma \ref{lemma:key}.

\begin{proof}[Proof of Lemma \ref{lemma:key}.]

We have 
\begin{align*}
B_{n+1}(0) &= -\frac{1}{n!} \sum_{k=1}^{n} \binom{n}{k} (H_{n}-H_{n-k}) (-1)^{k} \\
&= -\frac{1}{n!} \left (-H_n - \sum_{k=1}^n \binom{n}{k} (-1)^k H_{n-k} \right ) \\
&= -\frac{1}{n!} \left (-H_n - \left (\frac{(-1)^{n+1}}{n} -H_n \right ) \right ) \\
&= \frac{(-1)^{n+1}}{n \cdot n!}\, .
\end{align*}
\end{proof}

\begin{corollary} For $n\geq 2$,
$$
B'_{n+1}(0)=\frac{(-1)^n}{(n-1)(n-1)!}\, .
$$
\end{corollary}

\begin{proof}
From (\ref{eq:recurrence_B}) we have
$$
B'_{n+1}(0)= B_n(0),
$$
and the result follows from Lemma \ref{lemma:key}.
\end{proof}

This is related to the following identity with harmonic numbers:

\begin{lemma} For $n\geq 2$, we have
$$
\sum_{k=0}^n k \binom{n}{k} (-1)^k H_{n-k} =(-1)^{n-1}\frac{n}{n-1}\, .
$$
\end{lemma}

\begin{proof}
For $n\geq 2$, we have
$$
\sum_{k=0}^n k \binom{n}{k} (-1)^k a^{n-k} =\left . x\frac{d}{dx} (a+x)^n\right |_{x=-1} =-n(a-1)^{n-1}\, ,
$$
therefore
\begin{align*}
\sum_{k=0}^n k\binom{n}{k} (-1)^k H_{n-k} &= \int_0^{+\infty} \frac{ -n\left (1-1  \right )^{n-1} e^{-t}+n
\left (e^{-t}-1\right )^{n-1} e^{-t} }{1-e^{-t}} \, \dd t \\
&= (-1)^{n-1}n \int_0^{+\infty} (1-e^{-t})^{n-2}e^{-t} \, \dd t \\
&= (-1)^{n-1}n\int_0^1 x^{n-2} \, \dd x \\
&=(-1)^{n-1}\frac{n}{n-1}\, .
\end{align*}
\end{proof}

\section{Egyptian formulas for rational numbers}

We start with the simplest case: an Egyptian formula for rationals. The following is an 
``infinite Egyptian fraction decomposition'' for $\frac{1}{n}$.

\begin{proposition}[Infinite Egyptian fraction decomposition]  \label{prop:5.1}
For $n\geq 2$, we have
$$
\frac{1}{n}= \sum_{j=1}^{+\infty} \frac{1}{\binom{j+n+1}{n+1}}\, . 
$$
\end{proposition}

\begin{proof}
Notice that from Proposition \ref{prop:polar_regularization} we have
$$  
I_n(s)=  (-1)^n\int_0^{+\infty} \frac{1}{t^n} \left ( e^{-st}- P_n(s,t) e^{-t}\right ) \, \dd t  ,
$$
with
$$
P_n(s,t)=\sum_{k=0}^{n-1} \frac{(1-s)^k t^k}{k!}\, ,
$$
hence
$$
P_n(0,t)=\sum_{k=0}^{n-1} \frac{t^k}{k!}\, .
$$
So for $n\geq 2$, we can develop and exchange the integral and the summation:
\begin{align*}
(-1)^n I_n(0) &= \int_0^{+\infty} \frac{e^{-t}}{t^n} \left ( e^t-\sum_{k=0}^{n-1} \frac{t^k}{k!} \right ) \dd t \\
&=\int_0^{+\infty} e^{-t} \sum_{j=0}^{+\infty} \frac{t^j}{(j+n)!} \, \dd t \\
&=\sum_{j=0}^{+\infty} \frac{j!}{(j+n)!}\\
&= \frac{1}{n!} \sum_{j=0}^{+\infty} \frac{1}{\binom{j+n}{n}} \, .
\end{align*}

Now we have from Lemma \ref{lemma:key},
$$
I_n(0)=B_n(0)=\frac{(-1)^n}{(n-1)(n-1)!} \, ,
$$
thus
$$
\frac{n}{n-1}= \sum_{j=0}^\infty  \frac{1}{\binom{j+n}{n}}\, ,
$$
and the result follows.
\end{proof}

As one referee has pointed out to us, Proposition \ref{prop:5.1} follows also by a telescoping sum over 
 $$
  \frac{n}{\binom{j+n+1}{n+1}}=  \frac{j+n+1}{\binom{j+n+1}{n+1}}-  \frac{j+n+2}{\binom{j+n+2}{n+1}}\, ,
$$
which is found using Gosper's algorithm. We show here that this formula results from our general approach.

\section{BBP-like formulas for $\log s$}

In general we have
\begin{proposition} \label{prop:key} For $|s-1|< 1$, or $|s-1|= 1$ and $n\geq 2$, we have 
$$
I_n(s)=  \frac{(-1)^n}{n!}\sum_{j=0}^{+\infty} \frac{(1-s)^{j+n}}{\binom{j+n}{n}}\, .
$$
\end{proposition}

\begin{proof}
The condition  $\Re s>0$ ensures the convergence of the integrals and  $|s-1|< 1$, or $|s-1|= 1$ and $n\geq 2$
ensures the convergence of the series,
\begin{align*}
(-1)^nI_n(s) &=  \int_0^{+\infty} \frac{1}{t^n} \left ( e^{-st}- P_n(s,t) e^{-t}\right ) \, \dd t   \\
&=  \int_0^{+\infty} \frac{e^{-t}}{t^n} \left ( e^{(1-s)t}- P_n(s,t) \right ) \, \dd t   \\
&=  \int_0^{+\infty} \frac{e^{-t}}{t^n} \left ( e^{(1-s)t}- \sum_{k=0}^{n-1} \frac{(1-s)^k t^k}{k!} \right ) \, \dd t   \\
&=  \int_0^{+\infty} \frac{e^{-t}}{t^n} \, \sum_{j=0}^{+\infty} \frac{(1-s)^{j+n} t^{j+n}}{(j+n)!}  \, \dd t   \\
&=  \int_0^{+\infty} e^{-t} \sum_{j=0}^{+\infty} \frac{(1-s)^{j+n} t^{j}}{(j+n)!}  \, \dd t   \\
&=  \sum_{j=0}^{+\infty} \frac{(1-s)^{j+n} j!}{(j+n)!}    \\
&=  \frac{1}{n!}\sum_{j=0}^{+\infty} \frac{(1-s)^{j+n}}{\binom{j+n}{n}}  \, .
\end{align*}
\end{proof}

\begin{remark} \label{rem:conditional}
 The formula in Proposition \ref{prop:key} also holds for $|s-1|=1$ and $s\neq 0$, but the convergence of
the sum is only conditional. This can be checked by continuity of both sides making $|s-1|\to 1$.
\end{remark}

Now, we have
$$
I_n(s)= \frac{s^{n-1}}{(n-1)!} \log s +B_n(s),
$$
and since $B_n\in \QQ[s]$ we get, 
$$
I_n(s)= (-1)^n\sum_{j=0}^{+\infty} \frac{(1-s)^{j+n}}{n! \binom{j+n}{n}} =  \frac{s^{n-1}}{(n-1)!} \log s +B_n(s).
$$
In particular, for $s\in\QQ$ we have
$$
\sum_{j=0}^{+\infty} \frac{(1-s)^{j+n}}{n! \binom{j+n}{n}} \in \QQ \oplus \QQ \log s \, .
$$

\begin{theorem} \label{thm:log-formula}
 Let $|s-1|< 1$, or $|s-1|= 1$ and $n\geq 2$. Then we have
 $$
 \log s=- \frac{(n-1)!}{s^{n-1}} B_n(s) + (-1)^n
 \frac{(n-1)!}{s^{n-1}}  \sum_{j=0}^{+\infty} \frac{(1-s)^{j+n}}{n! \binom{j+n}{n}}\,  . 
$$ \hfill $\qed$
\end{theorem}

We get a group of formulas for $\log 2$ by specializing at $s=2$. We have
 $$
 \log 2=- \frac{(n-1)!}{2^{n-1}} B_n(2) +
 \frac{(n-1)!}{2^{n-1}}  \sum_{j=0}^{+\infty} \frac{(-1)^{j}}{(j+1)(j+2)\ldots (j+n)}\,  . 
 $$
Using the values $B_2(2)=-1$, $B_3(2)=-\frac12$, $B_4(2)=-\frac89$, $B_5(2)=-\frac{131}{240}$, $B_6(2)=
-\frac{661}{3600}$, we get:
\begin{align}
 \log 2 &= \frac12 +\frac12 \sum_{j=1}^\infty \frac{(-1)^{j+1}}{j(j+1)} \label{eqn:log2-1} \\
 \log 2 &= \frac58 +\frac12 \sum_{j=1}^\infty \frac{(-1)^{j+1}}{j(j+1)(j+2)} \label{eqn:log2-2} \\
 \log 2 &= \frac23 +\frac34 \sum_{j=1}^\infty \frac{(-1)^{j+1}}{j(j+1)(j+2)(j+3)} \label{eqn:log2-3} \\
 \log 2 &= \frac{131}{192} +\frac32 \sum_{j=1}^\infty \frac{(-1)^{j+1}}{j(j+1)(j+2)(j+3)(j+4)} \label{eqn:log2-4} \\
 \log 2 &= \frac{661}{960} +\frac{15}4 \sum_{j=1}^\infty \frac{(-1)^{j+1}}{j(j+1)(j+2)(j+3)(j+4)(j+5)} \label{eqn:log2-5} 
\end{align}

Specializing at $s=1/2$, we get the following formula for $n\geq 1$.
 $$
 \log 2=- \frac{(n-1)!}{2^{n-1}} B_n(1/2) +
 (-1)^{n} (n-1)!\sum_{j=0}^{+\infty} \frac{1}{2^{j+1}(j+1)(j+2)\ldots (j+n)}\,  . 
$$
Using the values
$B_2(1/2)=-2$, $B_3(1/2)=-1$, $B_4(1/2)=-\frac{40}{36}$, $B_5(1/2)=\frac{7}{18}$, $B_6(1/2)=
-\frac{47}{225}$, we get the formulas:
 \begin{align}
 \log 2 &= \sum_{j=1}^\infty \frac{1}{2^jj} \label{eqn:log2-6} \\
 \log 2 &= 1-\sum_{j=1}^\infty \frac{1}{2^jj(j+1)} \label{eqn:log2-7} \\
 \log 2 &= \frac12 +2\sum_{j=1}^\infty \frac{1}{2^jj(j+1)(j+2)} \label{eqn:log2-8} \\
 \log 2 &= \frac56- 6 \sum_{j=1}^\infty \frac{1}{2^jj(j+1)(j+2)(j+3)} \label{eqn:log2-9} \\
 \log 2 &= \frac7{12} +24 \sum_{j=1}^\infty \frac{1}{2^jj(j+1)(j+2)(j+3)(j+4)} \label{eqn:log2-10} \\
 \log 2 &= \frac{47}{60} -120\sum_{j=1}^\infty \frac{(-1)^{j+1}}{2^jj(j+1)(j+2)(j+3)(j+4)(j+5)} \label{eqn:log2-11} 
\end{align}
All these formulas appear in \cite{Wikipedia}.

It is customary to write the formulas above by splitting the denominators into simple fractions. For instance, the
fourth formula can be written as
$$
 \log 2 =\frac{5}{6}-\sum_{j=1}^{+\infty} \frac{1}{2^j} \left (\frac{1}{j}-\frac{3}{j+1}+\frac{3}{j+2}-\frac{1}{j+3} \right ) .
 $$
If we group for $j=4k,4k+1,4k+2,4k+3$, we get
\begin{align*}
 \log 2 =& \frac{5}{6}
-\sum_{k=1}^{+\infty} \frac{1}{2^{4k}} \left (\frac{1}{4k}-\frac{3}{4k+1}+\frac{3}{4k+2}-\frac{1}{4k+3} \right ) 
-\sum_{k=0}^{+\infty} \frac{1}{2^{4k}} \left (\frac{1/2}{4k+1}-\frac{3/2}{4k+2}+\frac{3/2}{4k+3}-\frac{1/2}{4k+4} \right ) \\
&-\sum_{k=0}^{+\infty} \frac{1}{2^{4k}} \left (\frac{1/4}{4k+2}-\frac{3/4}{4k+3}+\frac{3/4}{4k+4}-\frac{1/4}{4k+5} \right ) 
-\sum_{k=0}^{+\infty} \frac{1}{2^{4k}} \left (\frac{1/8}{4k+3}-\frac{3/8}{4k+4}+\frac{3/8}{4k+5}-\frac{1/8}{4k+6} \right ) \\
=& \frac23 +
 \sum_{k=1}^{+\infty} \frac{1}{2^{4k}} \left (\frac{1}{4k}+\frac{1/2}{4k+1}+\frac{1/4}{4k+2}+\frac{1/8}{4k+3} \right ) .
\end{align*}
We rewrite it in more classical form:
\begin{align} \label{eqn:log2-BBP}
 \log 2 =& \frac23 +\frac14
 \sum_{k=1}^{+\infty} \frac{1}{16^k} \left (\frac{8}{8k}+\frac{4}{8k+2}+\frac{2}{8k+4}+\frac{1}{8k+6} \right ) .
\end{align}

We can obtain many more binary BBP-like formulas. 
Specializing at $s=3/2$ we get the formula for $n\geq 1$,
$$
 \log (3/2)=-\left ( \frac23 \right )^{n-1} (n-1)! B_n(3/2) + 2
 (n-1)! 3^{n-1}  \sum_{j=0}^{+\infty} \frac{(-1)^j}{2^j (j+1)(j+2)\ldots (j+n)}\,  . 
$$
For instance, $n=4$ gives
 $$
 \log (3/2) =\frac{65}{162} + \frac{1}{216}\, \sum_{j=0}^{+\infty} \frac{(-1)^j}{2^j\binom{j+4}{j}} \, .
 $$
As before, the sum can also be written as
$$
\log (3/2) =\frac{65}{162} +\frac{1}{27}\sum_{k=1}^{+\infty} \frac{(-1)^{k+1}}{2^k} 
\left (\frac{1}{k}-\frac{3}{k+1}+\frac{3}{k+2}-\frac{1}{k+3} \right )  .
$$


In general, binary BBP formulas can be obtained from Theorem \ref{thm:log-formula}  by
taking $s=1 \pm \frac{1}{2^N}$,
 $$
 \log \frac{2^N\pm 1}{2^N}=\frac{-(n-1)!}{(1 \pm 2^{-N})^{n-1}} B_n \left(1 \pm 2^{-N}\right) + 
 \frac{(-1)^n(n-1)!}{(1 \pm 2^{-N})^{n-1}}  \sum_{j=0}^{+\infty} \frac{1}{2^{N(j+n)}(j+1)(j+2)\ldots (j+n)} . 
$$
Formulas of this sort are also obtained by Chamberland \cite{Ch}.

The numbers $2$ and $2^N\pm 1$, $N\geq1$, generate a multiplicative subgroup of $\QQ^*$, 
and for the elements $k$ in that subgroup, we have binary BBP formulas for $\log k$.
The first prime that it is not in this subgroup is $k=23$. Note that $2^{11}-1=23\cdot 89$, but 
these two primes appear always together in the factor decomposition of $2^N-1$ when $N$ is a multiple of $11$, and 
do not appear for other values of $N$. Also they do not appear at all in $2^N+1$, for any natural number $N$.
This can be checked as follows: first $2^{11}\equiv 1 \pmod{23}$, so
the order of $2$ in $\ZZ_{23}$ is $11$. In particular it cannot be that
$2^N\equiv -1 \pmod{23}$, since otherwise $2^{2N}\equiv 1\pmod{23}$, and
hence $2N|11$, so $N|11$ and thus $2^N\equiv 1 \pmod{23}$. On the
other hand, if $2^N\equiv 1 \pmod{23}$ then
$N$ is a multiple of $11$, and then $23\cdot 89 | (2^{11}-1) |(2^N-1)$.

\section{BBP-like formulas for $\pi$}

We may use Theorem \ref{thm:log-formula} for a complex value of $s$, then we can get BBP-formulas for $\log k$ and also
for $\pi$ separating real and imaginary parts. For $n=1$ (using Remark \ref{rem:conditional}), we have
$$
\log s = (s-1) \sum_{j=0}^{+\infty} \frac{(1-s)^j}{\binom{j+1}{1}}  = -  \sum_{j=1}^{+\infty} \frac{(1-s)^j}{j}\, ,
$$
which is the classical series for $\log s$.
Make $s=1+i$. We have $0<\Re (1+i) =1 <2$ and
$$
\log (1+i) = \log \sqrt{2} +i\frac{\pi}{4} =\frac12 \log 2 + i\frac{\pi}{4}\, ,
$$
and 
$$
\log(1+i) = i\sum_{j=0}^{+\infty} \frac{i^j}{\binom{j+1}{1}} \, .
$$
Separating real and imaginary part and $j=2k$ or $j=2k+1$ we get two BBP formulas, one for $\log 2$ and the other one for $\pi$:
 $$
 \log 2= 2\sum_{k=0}^{+\infty} \frac{(-1)^{k+1}}{\binom{2k+2}{1}}= \sum_{k=0}^{+\infty} \frac{(-1)^{k+1}}{k+1} 
 $$
and 
 $$
 \pi = 4\sum_{k=0}^{+\infty} \frac{(-1)^k}{2k+1} \, .
 $$
This last formula is just the first Machin formula for $\pi$, related to 
$$
\frac{\pi}{4} =\arctan 1 \, .
$$


For general $n\geq 2$, we take $s=1+i$, and we have
 \begin{align*}
 \log(1+i) &=\frac12 \log 2+i\frac\pi4 \\ 
 &= -\frac{(n-1)!}{(1+i)^{n-1}} B_n(1+i) + (-1)^n\frac{(1-i)^{n-1} (n-1)!}{2^{n-1}} 
 \sum_{j=0}^\infty \frac{(-i)^{j+n}}{(j+1)(j+2)\ldots (j+n)}.
 \end{align*}
Let 
 $$
 c_n=-\Im \left(\frac{(n-1)!}{(1+i)^{n-1}} B_n(1+i) \right),
 $$
so that
 \begin{align*}
\pi &= 4 c_n + 4(-1)^n \frac{(n-1)!}{2^{n-1}} \sum_{a=0}^{n-1}\binom{n-1}{a} 
\hspace{-2mm} \sum_{j\equiv n+a+1\,  (2)} 
 \frac{(-1)^{(j+n+a+1)/2}}{(j+1)(j+2) \ldots (j+n)} \, .
 \end{align*}
 
With this machinery at hand, we recover a number of known formulas.

\begin{proposition}[Leibniz] 
We have
$$
\pi= \frac{8}{3} +4\sum_{k=1}^{+\infty} \left (\frac{1}{4 k+1} -\frac{1}{4 k+3} \right ).
$$ 
\end{proposition}

\begin{proof}
We apply the above to $n=2$, where we have that $B_2(1+i)=-i$ and $c_2=-\Im (-i/(1+i))= 1/2$, thus
 \begin{align*}
\pi &= 2 + 2  \left( 
\sum_{j=0}^{+\infty} \frac{(-1)^j}{(2j+2)(2j+3)} 
+\sum_{j=0}^{+\infty} \frac{(-1)^j}{(2j+1)(2j+2)}\right)  \\ 
 &= 2 + 2  
\sum_{j=0}^{+\infty} (-1)^j \left( \frac{1}{2j+2} -\frac{1}{2j+3}+\frac{1}{2j+1}-\frac{1}{2j+2} \right) \\
&= 2 + 2  
\sum_{j=0}^{+\infty} (-1)^j \left( \frac{1}{2j+1} -\frac{1}{2j+3} \right) \\
&= 2 + 2  
\sum_{k=0}^{+\infty}  \left( \frac{1}{4k+1} -\frac{1}{4k+3} -\frac{1}{4k+3} +\frac{1}{4k+5}  \right) \\
&= \frac83 + 4 
\sum_{k=1}^{+\infty}  \left( \frac{1}{4k+1} -\frac{1}{4k+3} \right) .
 \end{align*}
\end{proof}

The original BBP formula from \cite{BBP} reads as follows:

\begin{theorem}[Bailey-Borwein-Plouffe]
We have
$$
 \pi = \sum_{k=0}^{+\infty} \frac{1}{16^k}\left (\frac{4}{8k+1} - \frac{2}{8k+4} -\frac{1}{8k+5} -\frac{1}{8k+6}\right ).
$$
\end{theorem}

\begin{proof}
Take $s=\frac{1+i}2$, so $\log s= -\frac12 \log 2 + i\frac{\pi}{4}$. Using the formula for $n=1$,
we have
 \begin{align*}
 \log \frac{1+i}2 &= -\sum_{j=1}^\infty \frac{(1-s)^j}{j} = -\sum_{j=1}^\infty \frac{(1-i)^j}{j 2^j} \, .
\end{align*}
Taking the imaginary part, and agroupping terms for $j=8k+r$, $r=1,2,\ldots, 7,8$, we get
 \begin{align*}
\frac{\pi}4 & = -  \sum_{k=0}^\infty \frac{1}{16^{k+1}} \left(
\frac{-8}{8k+1} -\frac{8}{8k+2}-\frac{4}{8k+3} +\frac{2}{8k+5}+\frac{2}{8k+6} + \frac{1}{8k+7}
\right)
\end{align*}
so
 \begin{equation}\label{eqn:11}
 \pi=   \sum_{k=0}^\infty \frac{1}{16^{k}} \left(
\frac{2}{8k+1} +\frac{2}{8k+2}+\frac{1}{8k+3} -\frac{1/2}{8k+5}-\frac{1/2}{8k+6} -\frac{1/4}{8k+7}
\right). 
 \end{equation}
Similarly, by taking the real part, we get
$$
 -\frac12 \log 2  = -\frac{71}{210} -  \sum_{k=1}^\infty \frac{1}{16^{k+1}} \left( \frac{16}{8k}+
\frac{8}{8k+1} -\frac{4}{8k+3}-\frac{4}{8k+4} -\frac{2}{8k+5} + \frac{1}{8k+7} \right)
 $$
so
 \begin{equation}\label{eqn:22}
\log 2  =  \sum_{k=0}^\infty \frac{1}{16^{k}} \left( \frac{2}{8k}+
\frac{1}{8k+1} -\frac{1/2}{8k+3}-\frac{1/2}{8k+4} -\frac{1/4}{8k+5} + \frac{1/8}{8k+7} \right).
 \end{equation}

Substracting (\ref{eqn:22}) and our previous formula (\ref{eqn:log2-BBP}), we get a null formula
\begin{align} \label{eqn:BBP-0}
 0=& 
\sum_{k=0}^{+\infty}  \frac{1}{16^{k}} \left( 
\frac{1}{8k+1}-\frac{1}{8k+2} -\frac{1/2}{8k+3}-\frac{1}{8k+4} -\frac{1/4}{8k+5} -\frac{1/4}{8k+6}+ \frac{1/8}{8k+7}  \right )
\end{align}
(note that the term $k=0$ gives exactly $1/105=71/105-2/3$).
Adding (\ref{eqn:11}) to twice this formula, we get
 $$
 \pi = \sum_{k=0}^{+\infty} \frac{1}{16^k}\left (\frac{4}{8k+1} - \frac{2}{8k+4} -\frac{1}{8k+5} -\frac{1}{8k+6}\right ).
$$
\end{proof}

In the proof we have proved and used the following null BBP formula that appears in \cite{BBP} :

\begin{proposition}\label{prop:zero} 
We have
\begin{equation}
\sum_{k=0}^{+\infty}  \frac{1}{16^{k}} \left( 
\frac{-8}{8k+1}+\frac{8}{8k+2} +\frac{4}{8k+3}+\frac{8}{8k+4} +\frac{2}{8k+5} +\frac{2}{8k+6}- \frac{1}{8k+7}  \right ) =0
\end{equation}

\end{proposition}

Null BBP formulas are very interesting and useful for rewritting BBP formulas.
They are obtained by comparing BBP formulas for the same number at different values of $s$.

\begin{proposition}\label{prop:zero} 
We have
$$
 \sum_{k=0}^\infty \frac{1}{2^{6k}}\left(\frac{16}{6k+1}-\frac{24}{6k+2}-\frac{8}{6k+3}-\frac{6}{6k+4}
+\frac{1}{6k+5}\right) = 0\, .
$$
\end{proposition}

\begin{proof}
 We use the formulas 
 \begin{align*}
 \log \frac32 &= -\sum_{k=1}^\infty \frac{1}{2^k} \frac{(-1)^k}{k}
= \sum_{k=0}^\infty \frac{1}{2^{6k}} \left(\frac{1/2}{6k+1}
+\frac{-1/4}{6k+2}+\frac{1/8}{6k+3}- \frac{1/16}{6k+4}+\frac{1/32}{6k+5}-\frac{1/64}{6k+6}\right), \\
 \log \frac34 &= -\sum_{k=1}^\infty \frac{1}{4^k} \frac{1}{k}
= -\sum_{k=0}^\infty \frac{1}{2^{6k}} 
\left(\frac{1/4}{3k+1} +\frac{1/16}{3k+2}+\frac{1/64}{3k+3}\right) =
-\sum_{k=0}^\infty \frac{1}{2^{6k}} \left(\frac{1/2}{6k+2}+\frac{1/8}{6k+4}+\frac{1/32}{6k+6}\right), \\
 \log \frac98 &= -\sum_{k=1}^\infty \frac{1}{8^{k}} \frac{(-1)^k}{k}=
 \sum_{k=0}^\infty \frac{1}{2^{6k}} \left(\frac{1/8}{2k+1}- \frac{1/64}{2k+2}\right) 
= \sum_{k=0}^\infty \frac{1}{2^{6k}} \left(\frac{3/8}{6k+3}- \frac{3/64}{6k+6}\right) .
\end{align*}
Adding the first two and substracting the third, we get 
$$
 \sum_{k=0}^\infty \frac{1}{2^{6k}}\left(\frac{1/2}{6k+1}-\frac{3/4}{6k+2}-\frac{1/4}{6k+3}-\frac{3/16}{6k+4}
+\frac{1/32}{6k+5}\right) = 0\, .
$$
and multiplying by $32$ we get the result.
\end{proof}

Finally, we include a proof of Bellard's formula.

\begin{theorem}[F. Bellard]
We have
$$
 \pi =\frac{1}{2^6} \sum_{k=0}^{+\infty} \frac{(-1)^k}{2^{10k}} 
 \left (-\frac{2^5}{4n+1} -\frac{1}{4n+3}+\frac{2^8}{10n+1}-\frac{2^6}{10n+3}
-\frac{2^2}{10n+5}- \frac{2^2}{10n+7}+\frac{1}{10n+9} \right ).
$$
\end{theorem}

\begin{proof}
We use the following factorization 
 $$
 1+i = \left( \frac{2+i}2\right)^2 \left(\frac{7+i}8\right)^{-1}\, ,
 $$
and taking imaginary parts
 $$
 \frac\pi4= 2\Im \log (1+i/2) - \Im \log((7+i)/8)\, .
$$

For $s=(7+i)/8$ and $n=1$, we get
 \begin{align}
 \Im \log((7+i)/8) =& -\Im \sum_{j=1}^\infty \frac{(1-i)^j}{j8^j}  \nonumber\\
 =&  \sum_{k=0}^\infty \frac{1}{2^{20 k}} \left(  
 \frac{1/8}{8k+1} + \frac{2/8^2}{8k+2} + \frac{2/8^3}{8k+3} 
- \frac{4/8^5}{8k+5}- \frac{8/8^6}{8k+6}- \frac{8/8^7}{8k+7} \right) \nonumber\\
=& \frac{1}{256}\sum_{l=0}^\infty \frac{(-1)^l}{2^{10l}} \left(
 \frac{32}{4l+1} + \frac{8}{4l+2} + \frac{1}{4l+3} \right), \label{eqn:7/8}
\end{align}
writing $j=8k+r$, $r=1,2,\ldots, 8$, and then $2k=l$.

Now take $s=1+i/2$ and $n=1$, to get
 \begin{align}
 \Im \log(1+i/2) =& -\Im \sum_{j=1}^\infty \frac{(-i)^j}{j2^j}
=\sum_{k=0}^\infty \frac{1}{2^{2k+1}} \frac{(-1)^k}{2k+1} \nonumber\\
 =& \frac{1}{256}\sum_{l=0}^\infty \frac{(-1)^l}{2^{10l}} \left(
\frac{128}{10l+1}-
\frac{32}{10l+3}+
\frac{8}{10l+5}-
\frac{2}{10l+7}+
\frac{1/2}{10l+9}\right). \label{eqn:1/2}
\end{align}
We substract twice (\ref{eqn:1/2}) minus (\ref{eqn:7/8}),
and use that  $2\frac{8}{10l+5}- \frac{8}{4l+2} =- \frac{4}{10l+5}$. Then we get the result.
\end{proof}

\section{On the classical BBP form}

As defined in \cite{BC} the classical BBP form is 
$$
Q(b,d,m,A)=\sum_{k=0}^{+\infty} \frac{1}{b^k}\sum_{l=1}^m \frac{a_l}{(km+l)^d}\, ,
$$
where $b,d,m$ are integers and $A=(a_1, a_2,\ldots , a_m)$ is a vector of integers. The degree is $d$ and the base is $b$.
Let us check that with our formula from Theorem \ref{thm:log-formula} we get BBP formulas of degree $1$.

\begin{lemma}
We have
$$
\frac{1}{\binom{j+n}{n}} = \sum_{l=1}^n \frac{c_l}{j+l}\, ,
$$
where for $l=1,2,\ldots ,n$, $c_l$ is an integer given by
$$
c_l = (-1)^{l-1} n \, \binom{n-1}{l-1} \, .
$$
\end{lemma}

\begin{proof}
As usual, multiply by $j+l$ and set $j=-l$ to get
$$
c_l = \frac{n!}{(n-l)(n-l-1)\cdots 2\cdot 1\cdot (-1)(-2)\cdots (-(l-1))} =(-1)^{l+1}l \, \binom{n}{l}=(-1)^{l-1} n \, \binom{n-1}{l-1}.
$$
\end{proof}

We have a general reorganization Lemma that shows that any sum of BBP form with more than $m$ fractions can be reorganized into one with $m$ terms.

\begin{lemma}
We have
$$
\sum_{j=0}^{+\infty} b^{-j} \left (\sum_{i=1}^n \frac{c_i}{(mj+i)^d} \right ) = \sum_{k=0}^{+\infty} b^{-k} \left (\sum_{l=1}^m \frac{a_l}{(km+l)^d} \right ) \, ,
$$
with 
$$
a_l =\sum_i c_i b^{\frac{i-l}{m}}\, .
$$ 
where the sum extends over indexes $l+1\leq i\leq n$ such that $mj+i=mk+l$.

\end{lemma}

\begin{proof}
For $k=0,1,\ldots$ and $l=1,\ldots m$ group the fractions of the sum modulo $m$ with $mj+i=mk+l$.
\end{proof}

These two Lemma prove that the BBP formulas that we get from Theorem \ref{thm_main} are of type
$Q(1,b,1, (a_1))$.

We can apply this reorganization to the summation in the formula from Theorem \ref{thm_main} 
and get (regrouping the terms with $j=k-l+1$
in the third equality),
 \begin{align} \label{eqn:end} 
  \sum_{j=0}^{+\infty} \frac{(1-s)^{j+n}}{\binom{j+n}{n}} \nonumber
 =& \sum_{j=0}^{+\infty} \sum_{l=1}^n \frac{c_l}{j+l} (1-s)^{j+n}\,   \nonumber\\
 =&\sum_{k=0}^{+\infty} \sum_{l=1}^n \frac{c_l}{k+1} (1-s)^{k+1+n-l} - 
\sum_{\substack{0\leq k\leq l-2 \\ l \leq n}} \frac{c_l}{k+1}(1-s)^{k+1+n-l} \nonumber\\
 =& -\sum_{\substack{0\leq k\leq l-2 \\ l \leq n}} \frac{c_l(1-s)^{k+n+1-l}}{k+1}
+   \sum_{k=0}^{+\infty} \frac{a_k}{k+1} (1-s)^{k+1} \, .
\end{align}
with $a_k=\sum\limits_{l=1}^n c_l (1-s)^{n-l}$. But we have
$$
a_k=\sum\limits_{l=1}^n c_l (1-s)^{n-l} =
\sum\limits_{l=1}^n (-1)^{l-1} n \, \binom{n-1}{l-1} (1-s)^{n-l} =  (-1)^{n-1} n (1-(1-s))^{n-1}=(-1)^{n-1} n s^{n-1}\, ,
$$


Hence, we recognize in the last sum of (\ref{eqn:end}) $\log s$, so the formula 
in Theorem  \ref{thm_main} for $n\geq 2$ is a rearrangement of the formula for $n=1$ that is the classical 
Taylor formula for $\log s$
 \begin{equation}\label{eqn:end2}
 \log s= - \sum_{k=0}^\infty \frac{(1-s)^{k+1}}{k+1}\, .
 \end{equation}


We can use this rearrangement to recover directly the formula for the polynomials $B_n$ directly:
\begin{align*}
\sum_{m\geq 1} &\frac{(1-s)^{n+m}}{m(m+1)\cdots (m+n)}  = \frac{(-1)^n}{n!}\sum_{m\geq 1} (1-s)^{m+n} \sum_{k=0}^n(-1)^k\binom{n}{k} \frac{1}{m+k}\\
& = \frac{(-1)^n}{n!} \sum_{k=0}^n(-1)^k\binom{n}{k}(1-s)^{n-k}\Big(\log(s)-\sum_{i=1}^k\frac{(1-s)^i}{i}\Big)
\end{align*}
but
\begin{equation*}
A_n(s)=\frac{(-1)^n}{n!} \sum_{k=0}^n(-1)^k\binom{n}{k}(1-s)^{n-k}=(1-(1-s))^n=\frac{(-1)^n}{n!}s^n
\end{equation*}
and
\begin{equation*}
B_n(s)= \frac{(-1)^n}{n!} \sum_{k=0}^n(-1)^k\binom{n}{k}(1-s)^{n-k}\sum_{i=1}^k \frac{(1-s)^i}{i}
\end{equation*}
which gives after some rearrangment the expression for $B_n(s)$

Formally, there is no extra content in the formulas for the same parameter $s$ but different integers $n\geq 2$. However, these
rearrangements are computationally useful, and they are not easy to produce. The iterated integrals $I_n(s)$ or Proposition \ref{prop:end4}
gives a systematic method to find a family of such resummations. The expression in terms of combinatorical coefficients in the denominator 
that arise by the iterated integrals in this type of sums can present sometimes some advantages. Of course one is inmediately reminded 
(even if it is a formula of higher degree) of the famous Apery sum for $\zeta(3)$  starting point of his proof of the irrationality of this 
number.

\section*{Appendix. Location of the zeros of the polynomials $B_n$}

The application of the formula in Theorem \ref{thm:log-formula} to roots of $B_n$, in particular to real roots, 
gives BBP-like formulas of a special form. We study the location of the roots of $B_n$ and the number of real roots.

To understand the polynomials $B_n(s)$, we 
introduce the polynomials $C_n(x)$ of degree $n-2$, for $n\geq 2$, defined by
\begin{equation}\label{eq_for_Bn}
B_n(s)=-\frac{1}{(n-1)!} (s-1)C_n(s-1),
\end{equation}
so that by Proposition \ref{prop:B_polynomials} 
\begin{equation}\label{eq_for_C}
C_n(x)=\sum_{k=0}^{n-2} \binom{n-1}{k+1} (H_{n-1}-H_{n-k-2}) \, x^{k} \, .
\end{equation}

We list the polynomials:
 \begin{align*} 
 B_1(s) &=0,  \nonumber\\
 B_2(s) &= -(s-1) , \nonumber\\
 B_3(s) & =-\frac14 (s-1) (3s-1) ,  \nonumber\\
 B_4(s) &=-\frac{1}{6}(s-1) \left (1+\frac52 (s-1)+\frac{11}{6} (s-1)^2 \right ) , \nonumber \\
 B_5(s) &= -\frac{1}{20}(s-1) \left (1+\frac72 (s-1)+\frac{13}{3} (s-1)^2 + \frac{25}{12} (s-1)^3 \right ) , \nonumber\\
B_6(s) &=-\frac{1}{120}(s-1) \left (1+\frac92 (s-1)+\frac{47}{6} (s-1)^2 + \frac{77}{12} (s-1)^3+\frac{137}{60} (s-1)^4  
\right ) ,  \nonumber\\
B_7(s) &=-\frac{1}{740}(s-1) \left (1+\frac{11}{2}(s-1)+\frac{37}{3} (s-1)^2 + 
 \frac{57}{4} (s-1)^3+\frac{87}{10} (s-1)^4+\frac{49}{20} (s-1)^5\right ) , \nonumber
\end{align*}
and accordingly,
 \begin{align*}
C_2(x) &= x ,\\
C_3(x) &= \frac12 (3x+2), \\
C_4(x) &= \frac{1}{6}(6+15 x +11 x^2), \\
C_5(x) &= \frac{1}{12} \left (12+42 x +52 x^2 +25 x^3 \right ), \\
C_6(x) &= \frac{1}{60} \left (60+180 x +470 x^2 +385 x^3 + 137 x^4 \right ),\\
C_7(x) &= \frac{1}{60} \left (60+330 x +740 x^2 +855 x^3 + 522 x^4 + 147 x^5 \right ).
\end{align*}

We want to locate the zeros of $C_n(x)$.

\begin{lemma} \label{lem:4.1}
We have
\begin{align*}
C_n(0) &= 1, \\
C_n(-1) &=(n-1)!B_n(0)=\frac{(-1)^n}{n-1} \, .
\end{align*}
\end{lemma}

\begin {proof}
The value at $x=-1$ follows from Lemma \ref{lemma:key}. The value at $x=0$ by (\ref{eq_for_C}).
\end {proof}

Let 
 $$
 D_n(x)=x\, C_n(x)=\sum_{k=0}^{n-1} \binom{n-1}{k} (H_{n-1}- H_{n-k-1}) x^k \, .
 $$
The zeros of $D_n$ are those of $C_n$ and an extra zero at $x=0$.
Now we have two interesting equalities:
 \begin{align}\label{eqn:1}
 D_n'(x) =& \sum_{k=1}^{n-1} k\binom{n-1}{k} (H_{n-1}- H_{n-k-1}) x^{k-1} \nonumber\\
 =& \sum_{k=1}^{n-1} (n-1)\binom{n-2}{k-1} (H_{n-1}- H_{n-k-1}) x^{k-1} \nonumber\\
 =& \sum_{k=0}^{n-2} (n-1)\binom{n-2}{k} \left(\frac{1}{n-1}+ H_{n-2}- H_{n-k-2}\right) x^{k} \nonumber\\
 =& \sum_{k=0}^{n-2} (n-1)\binom{n-2}{k} ( H_{n-2}- H_{(n-1)-k-1}) x^{k} + \sum_{k=0}^{n-2} \binom{n-2}{k} x^k \nonumber\\
 =& (n-1) D_{n-1}(x)+ (1+x)^{n-2} \, ,
\end{align}
and 
\begin{align}\label{eqn:2}
 (1+x) D_n'-& (n-1) D_n = (n-1) (1+x) D_{n-1} - (n-1) D_n +(1+x)^{n-1} \nonumber\\
  =&\sum_{k=0}^{n-2} (n-1)\binom{n-2}{k} (H_{n-2}- H_{n-k-2}) x^k +
   \sum_{k=0}^{n-2} (n-1)\binom{n-2}{k} (H_{n-2}- H_{n-k-2}) x^{k+1} \nonumber\\
   &-\sum_{k=0}^{n-1} (n-1)\binom{n-1}{k} (H_{n-1}- H_{n-k-1}) x^k + (1+x)^{n-1} \nonumber\\
  =&\sum_{k=0}^{n-2} (n-1)\binom{n-2}{k} \left(H_{n-2}- H_{n-k-1} + \frac{1}{n-k-1}\right) x^k \nonumber\\ &+
   \sum_{k=1}^{n-1} (n-1)\binom{n-2}{k-1} (H_{n-2}- H_{n-k-1}) x^{k} \nonumber\\
   &-\sum_{k=0}^{n-1} (n-1)\binom{n-1}{k} \left(\frac{1}{n-1}+H_{n-2}- H_{n-k-1}\right) x^k + (1+x)^{n-1} \nonumber\\
  =&\sum_{k=0}^{n-2} \frac{n-1}{n-k-1}\binom{n-2}{k} x^k 
   -\sum_{k=0}^{n-1} \binom{n-1}{k} x^k + (1+x)^{n-1} \nonumber\\
  =&\sum_{k=0}^{n-2} \binom{n-1}{k} x^k 
   -\sum_{k=0}^{n-1} \binom{n-1}{k} x^k + (1+x)^{n-1} \nonumber\\
 =& (1+x)^{n-1}-x^{n-1} = Q_n(x).
\end{align}

Using these equalities, we can prove the following:

\begin{proposition}
For $n\geq 2$ even, the polynomial $C_n(x)$ has no real roots.

For $n\geq 3$ odd, the polynomial $C_n(x)$ has only one real root and it lies in the interval $]\!-1,0[$.
\end{proposition}

\begin{proof}
We want to prove by induction that:
\begin{itemize}
 \item For $n$ even, $x=0$ is the only (simple) zero of $D_n(x)$. And $D_n(x)<0$ for
$x<0$ and $D_n(x)>0$ for $x>0$. 
 \item For $n$ odd, $D_n$ has two zeros, at some $x_0\in ]\!-1,0[$ 
and at $x=0$. And $D_n(x)>0$ for $x\in ]\!-\infty,x_0[\,\,\cup \,\, ]0,\infty[$ and $D_n(x)<0$ for
$x\in ]x_0,0[$.
\end{itemize}

Let $n$ be even. We want to prove that $D_n(x)$ has only a zero at $x=0$.
Note that $D_n(0)=0$ and $D_n'(0)=1$, so $D_n$ is increasing at $x=0$.
For $n$ even we have $Q_n(x)>0$ everywhere.
\begin{itemize}
\item If $x\leq -1$ then $D_{n-1}(x)<0$ by induction hypothesis. By (\ref{eqn:1})
we have $D_n'(x)>0$, so it is increasing there. By Lemma \ref{lem:4.1}, $D_n(-1)<0$ so there are no
zeros on $]\!-\infty,-1]$. 
\item If $x> 0$ then $D_{n-1}(x)>0$ by induction hypothesis. By (\ref{eqn:1})
we have $D_n'(x)>0$, so it is increasing there. As $D_n(0)=0$, there are no
zeros on $]0,\infty[$. 
\item If $x\in ]\!-1,0]$ then $Q_n(x)>0$. If $D_n(x)=0$ then
(\ref{eqn:2}) says that $(1+x)D_n'(x)>0$. So $D_n$ is increasing at every zero.
As $x=0$ is a zero, then this implies that there is only one zero of $D_n$.
\end{itemize}

Now let $n$ be odd. We want to prove that $D_n(x)$ has a zero at some $x_0\in ]\!-1,0[$
and at $x=0$, it is positive on $]\!-\infty,x_0[\,\, \cup\,\, ]0,\infty[$ and negative at $]x_0,0[$.
Note that $D_n(0)=0$ and $D_n'(0)=1$, so it is increasing at $x=0$.
Note that for $n$ odd we have $Q_n(x)>0$ for $x>-1/2$, and
$Q_n(x)<0$ for $x<-1/2$.
\begin{itemize}
\item If $x\leq -1$ then $D_{n-1}(x)<0$ by induction hypothesis. By (\ref{eqn:1})
we have $D_n'(x)<0$, so it is decreasing there. By Lemma \ref{lem:4.1} $D_n(-1)>0$, so there are no
zeros on $]\!-\infty,-1]$. 
\item If $x> 0$ then $D_{n-1}(x)>0$ by induction hypothesis. By (\ref{eqn:1})
we have $D_n'(x)>0$, so it is increasing there. As $D_n(0)=0$, there are no
zeros on $]0,\infty[$. 
\item If $x\in ]\!-1/2,0[$ then (\ref{eqn:2}) says that $(1+x)D_n'(x) > (n-1) D_n(x)$.
So if there is a zero, $D_n$ is increasing. As the last zero before $x=0$ cannot
be increasing, this last zero has to be $x_0\leq -1/2$.
\item For $x=-1/2$, if it was a zero of $D_n$, then it is also a zero of $D_n'$ 
because of (\ref{eqn:2}). Then we write $x=-1/2+h$, and develop (\ref{eqn:2})
to see that $D_n'(-1/2+h)>0$ for $h>0$ small. But this implies that there
must be another zero of $D_n$ in $]\!-1/2,0[$ with decreasing slope, which contradicts 
the previous item.
\item If $x\in ]\!-1,-1/2[$ then (\ref{eqn:2}) says that $(1+x)D_n'(x) < (n-1) D_n(x)$.
So if there is a zero, $D_n$ is decreasing. There must be at least one zero,
but there cannot be two zeros, since there cannot be two decreasing consecutive
zeros.
\end{itemize}
\end{proof}


It is relevant to locate the complex zeros of $C_n(x)$.
The polynomial $C_4$ has a pair of conjugate complex roots
$x\approx-0.68182 \pm 0.28386 i$. 
The polynomial $C_5$ has one real root $x_0\approx-0.61852$ and a pair of conjugate complex roots
$x\approx -0.73074 \pm 0.49200 i$. The polynomial $C_6$ has $2$ pairs of conjugate complex roots:
$x \approx-0.18252 \pm 0.39103 i$, $x \approx-1.2226 \pm 0.9258 i$. 
We may expect that all roots of $C_n(x)$ have $\Re x\in ]\!-\infty,0[$.

To locate the complex roots of $C_n(x)$, we rewrite  the differential equation (\ref{eqn:2}) as
 $$
 \left((1+x)^{-(n-1)} D_n(x)\right)' = (1+x)^{-n} Q_n(x).
 $$
Take $f_n(x)=(1+x)^{-(n-1)} D_n(x)$, hence $\dd f_n =(1+x)^{-n} \left((1+x)^{n-1}-x^{n-1}\right) \dd x$. 
We make the change of variables $w= \frac{x}{1+x}$ to get 
$\dd f_n = \frac{1-w^{n-1}}{1-w} \dd w = (1+w+....+w^{n-2}) \dd w$, and integrating
 $$ 
 f_n= w+\frac12 w^2 +.... + \frac{1}{n-1}w^{n-1} \, ,
 $$
where we have used that for $w=0$, it is $x=0$ and hence $f_n=0$.
Note that $f_n(w)$ is the truncation of the series $-\log(1-w)$, which is convergent 
on $|w|<1$. 

\begin{proposition}\label{prop:roots}
The polynomial $f_n(w)$ has no roots in $|w|\leq 1$ except $w=0$.
\end{proposition}

\begin{proof}
We will look at the polynomial 
 $$
 Q(w)=f_n(w)(1-w)/w= 1- \sum_{k=1}^{n-2} \frac{1}{k(k+1)} w^k - \frac{1}{n-1}w^{n-1}\, ,
 $$
for which we want to check that the only root in the disc $|w|\leq 1$ is $w=1$.
For $|w|\leq 1$, we have
 $$
 \left| \sum_{k=2}^{n-2} \frac{1}{k(k+1)} w^k + \frac{1}{n-1}w^{n-1}\right| \leq
 \sum_{k=2}^{n-2} \frac{1}{k(k+1)} + \frac{1}{n-1}w^{n-1} =\frac12 \, .
 $$
Then if $Q(w)=0$, we have 
 $$
 \left|1-\frac12 w \right| \leq \frac12
 $$
which implies $|w-2|\leq 1$. Combined with $|w|\leq 1$, we have $w=1$.
\end{proof}

Undoing the change of variables $w= \frac{x}{1+x}$, we get that all
roots of $C_n(x)$ are in $\Re x < -\frac12$. Therefore, with (\ref{eq_for_Bn})
we get that the roots of $B_n(s)$ are $s=1$ and the others lie in $\Re s < \frac12$.


\begin{thebibliography}{99}


\bibitem{B} BAILEY, D.H.; {\it A compendium of BBP-type formulas for mathematical constants}, 
\url{https://crd-legacy.lbl.gov/~dhbailey/dhbpapers/bbp-formulas.pdf}, accessed August 2017.

\bibitem{BB} BAILEY, D.H.; BORWEIN, P.B.; {\it Experimental Mathematics: Recent Developments and Future Outlook}, Mathematics Unlimited — 2001 and Beyond, Springer, p.51-66, 2001.

\bibitem{BB2} BAILEY, D.H.; BORWEIN, P.B.; {\it Experimental Mathematics: Examples, methods and implications}, Notices AMS, p.502-514, 2005.

\bibitem{BBMW}
BAILEY, D.H.; BORWEIN, J.M.; MATTINGLY, A.; WIGHTWICK, G.; {\it The computation of previously inaccessible digits of $\pi^2$ and Catalan's constant},
Notices Amer. Math. Soc. {\bf 60}, 7, p. 844-855, 2013.

\bibitem{BBP} BAILEY, D.H.; BORWEIN, P.B.; PLOUFFE, S.; {\it On the Rapid Computation of Various Polylogarithmic Constants}, 
Mathematics of Computation, \textbf{66}, 218, p. 903-913, 1997.

\bibitem{BC} BAILEY, D.H.; CRANDALL, R.E.; {\it On the random character of fundamental constant expansions}, Experimental Mathematics, \textbf{10}, 2, p. 175-190, 2001.

\bibitem{BC2} BAILEY, D.H.; CRANDALL, R.E.; {\it Random generators and normal numbers}, 
Experimental Mathematics, \textbf{11}, p. 527-546, 2002.

\bibitem{BM} BAILEY, D.H.; MISIUREWICZ, M.; {\it A strong hot spot theorem}, 
Proc. Amer. Math. Soc. \textbf{134}, p. 2495-2501, 2006.


\bibitem{Be} BELLARD, F.;  {\it Pi formulas, algorithms and computations}, \url{https://bellard.org/pi/}, accessed March 2019.

\bibitem{Br} BROADHURST, D.J.;  {\it Polylogarithmic ladders, hypergeometric series
and the ten millionth digits of $\zeta(3)$ and $\zeta(5)$}, ArXiv:9803067, 1998.


\bibitem{Bo} BOREL, \'E.;  {\it Les probabilit\'es d\'enombrables et leurs applications arithm\'etiques}, 
Rend. Circ. Mat. Palermo, \textbf{27}, p. 247-271, 1909.

\bibitem{BM} BOROS, G.; MOLL, V.H.; {\it Irresistible integrals}, Cambridge Univ. Press., 2004.

\bibitem{BBo} BORWEIN, J.M.; BORWEIN, P.B.; {\it Pi and the AGM – A Study in Analytic Number Theory and Computational Complexity}, 
Wiley, 1987. 

\bibitem{BBG} BORWEIN, J.M.; BORWEIN, D.; GALWAY, W.F.; {\it Finding and excluding b-ary Machin-type individual digit
formulae}, Canadian J. Math. \textbf{56}, p. 897-925, 2004.

\bibitem{Ch} CHAMBERLAND, M.; {\it Binary BBP-formulae for logarithms and generalized Gaussian-Mersenne primes},
J. Integer Seq., \textbf{6}, 3, Article 03.3.7, 10 pp., 2003.

\bibitem{CC} CHUDNOVSKY, D.V.; CHUDNOVSKY, G.V.; {\it Approximation and complex multiplication according to Ramanujan}, 
Ramanujan revisited: Proceedings of the centennial conference (G.E. Andrews et al., eds), Academic Press, Boston, p. 375-472, 1988.


\bibitem{L} LAGARIAS, J.C.; {\it On the normality of arithmetical constants},  Experimental Math., \textbf{10}, 3, p. 355-368, 2001.

\bibitem{M} MATHAR, R.; {\it Series of reciprocal powers of $k$-almost primes},  ArXiv:0803.0900, 2008.

\bibitem{MPM1} MU\~NOZ, V.; P\'EREZ-MARCO, R.; {\it Unified Treatment of Explicit and Trace Formulas via Poisson-Newton formula},  
Comm. Math. Phys., \textbf{336}, 3, p. 1201-1230, 2015.

\bibitem{MPM2} MU\~NOZ, V.; P\'EREZ-MARCO, R.; {\it Weierstrass constants and exponential periods}, in preparation, 2019.


\bibitem{MMR} MEDINA, L.A.; MOLL, V.H.; ROWLAND, E.S. {\it Iterated primitives of logarithmic powers}, International Journal of Number Theory, 
\textbf{7}, 3, p.623-634, 2011.

\bibitem{Q} QUEFF\'ELEC, M.; {\it Old and new results on normality},  IMS Lecture Notes, Monographs S., \textbf{48},  p. 225-236, 2006.

\bibitem{R} RAMANUJAN, S.; {\it Modular equations and approximations to $\pi$}, Quart. J. Math. (Oxford), \textbf{45},  p. 350-372, 1914.

\bibitem{Sto1} ST\"ORMER, C.; {\it Sur l'application de la th\'eorie des nombres entiers complexes \`a la solution 
en nombres rationnels $x_1, x_2, \ldots x_n, c_1, c_2,\ldots , c_n, k$ de l'\'equation: $c_1 \arctan x_1 + c_2 \arctan x_2 + 
\ldots +c_n \arctan x_n=k\frac{\pi}{4}$}, Arkiv for Mathematik Og Naturvidenskab, \textbf{XIX},\textbf{3},  p. 3-95, 1896.

\bibitem{Sto2} ST\"ORMER, C.; {\it Solution compl\`ete en nombres entiers de l'\'equation 
$m\arctan \frac1x + n\arctan \frac1y=k\frac{\pi}{4}$}, Bull. Soc. Math. France, \textbf{27},  p. 160-170, 1899.

\bibitem{T} TAM, M.; {\it BBP-type Formula Database}, \url{https://bbp.carma.newcastle.edu.au/}, accessed March 2019.

\bibitem{Wikipedia} WIKIPEDIA; {\it Natural logarithm of $2$}, \url{https://en.wikipedia.org/wiki/Natural_logarithm_of_2}

\end{thebibliography}
\end{document}